\documentclass[11pt,a4paper]{article}
\usepackage[latin1]{inputenc}
\usepackage[english]{babel}
\usepackage{mathrsfs}
\usepackage{amsmath}
\usepackage{amsfonts}
\usepackage{amssymb}
\usepackage{amsthm}
\usepackage{graphicx, color}

\newtheorem{te}{Theorem}

\newtheorem{os}{Remark}
\newtheorem{prop}{Proposition}

\newtheorem{coro}{Corollary}
\newtheorem{ass}{Assumption}

\numberwithin{equation}{section}

\allowdisplaybreaks

\author{Mirko D'Ovidio\footnote{Dipartimento di Scienze di Base e Applicate per l'Ingegneria, A. Scarpa 10 - 00161, Rome, Italy. Email: mirko.dovidio@uniroma1.it. }\\ \emph{Sapienza University of Rome}}

\begin{document}

\title{Coordinates changed random fields on the sphere}

\maketitle

\begin{abstract}
We construct time dependent random fields on the sphere through coordinates change and subordination and we study the associated angular power spectrum. Some of this random fields arise naturally as solutions of partial differential equations with random initial condition represented by a Gaussian random field.   
\end{abstract}




\section{Introduction and main results}

In this paper we construct random fields on the unit sphere with different covariance structures by means of the coordinates change of Gaussian random fields given by the  subordinate Brownian motions on the sphere. We find out random fields with spectral representations on the unit sphere which are solutions to partial differential equations, with random initial condition, involving the fractional operator
\begin{equation*}
\mathbb{D}^\alpha_M\, f(x) = \int_{\mathbf{S}^2_1}\left( f(y) - f(x) \right) J(x,y)\lambda(dy) 
\end{equation*}
where $\lambda$ is the Lebesgue measure on the unit sphere $\mathbf{S}^2_1$. Such random fields turn out to be interesting because of their covariance structure and, in particular they well explain certain forms of the angular power spectrum studied in literature. In our view, the present work directs one's efforts towards  the analysis of the Cosmic Microwave Background (CMB) radiation which is currently at the core of physical and cosmological research (\cite{Dodelson, KolbTurn, HuWhite-2004}). Indeed, our model well accords with the results in \cite{Baldi-all-2008, Dom-ann-2006, DomPec-2008, DomPec-matphys-2010} where the authors deal with high frequency (or high resolution) analysis of random fields and, in particular with \cite{DomPec-matphys-2010} where a special class of models depending on the angular power spectrum has been introduced and relationship between ergodicity and asymptotic Gaussianity has been investigated. The mathematical objects we introduce here are time dependent random fields which well depict the random movement of  particles in random environments. In particular, if $T(\vartheta, \varphi)$ is a random field on the two dimensional  sphere, then as time passes, 
\begin{equation*}
\left( \vartheta_t, \varphi_t, T(\vartheta_t, \varphi_t) \right) \in \mathbf{R}^3, \quad t>0
\end{equation*}
is a randomly varying point in a random environment given by the stochastic sphere $(\vartheta_t, \varphi_t)$, $t\geq 0$, which is the rotational Brownian motion or Brownian motion on the $2$-sphere. Since the processes we deal with are random solutions of Cauchy problems, we also know their evolution in time. The angular power spectrum of the random fields considered in this work exhibits polynomial and/or exponential behaviour in the high frequency (or resolution) analysis and therefore, we introduce a large class of models in which many aspects can be captured, such as Sachs-Wolfe effect or Silk damping effect for instance.

\subsection{Statement of results}

\subsubsection{Preliminaries}

We focus on the zero-mean, isotropic Gaussian random fields $T$ on the sphere with spectral representation
\begin{equation}
T(x) = \sum_{l=0}^{\infty} \sum_{m=-l}^{+l} a_{lm}Y_{lm}(x) = \sum_{l=0}^{\infty} T_l(x) \label{Trep}
\end{equation}
where
\begin{equation}
a_{lm}= \int_{\mathbf{S}^2} T(x) Y_{lm}^*(x) \lambda(dx) \label{alm-coeff-intro}
\end{equation}
are Fourier random coefficients. We consider the square integrable $n$-weakly isotropic (with rotationally invariant law) random field $\{ T(x);\, x \in \mathbf{S}^2_1 \}$ on the sphere $\mathbf{S}^2_1 = \{ x \in \mathbf{R}^3 \, :\, |x|=1 \}$ for which 
\begin{equation*} \mathbb{E}T(gx) = 0, \quad \textrm{and} \quad T(gx) \stackrel{law}{=} T(x), \; \textrm{ for all } g \in SO(3) 
\end{equation*}
where $\overset{law}{=}$ represents equality in law of stochastic processes and $SO(3)$ the special group of rotations in $\mathbf{R}^{3}$. The random field $T$ is said to be $n$-weakly isotropic if $\mathbb{E}|T(gx)|^n < \infty$ ($n \geq 2$) for every $x \in \mathbf{S}^2_1$ and if, for every $x_1, \ldots , x_n \in \mathbf{S}^2_1$ and every $g \in SO(3)$ we have that
\begin{equation*}
\mathbb{E}[T(x_1) \times \cdots \times T(x_n)] = \mathbb{E}[T(gx_1) \times \cdots \times T(gx_n)].
\end{equation*}
Convergence in \eqref{Trep} holds in the mean square sense, both with respect to  $L^{2}(dP\otimes \lambda (dx))$ and with respect to $L^{2}(dP)$ for fixed $x \in \mathbf{S}^2_1$, $\lambda(dx)$ is the Lebesgue measure on the unit sphere $\mathbf{S}^2_1$ and the set of spherical harmonics $\{Y_{lm}:\, l \geq 0,\; m=-l, \ldots , +l\}$ represents an orthogonal basis for the space $L^2(\mathbf{S}^2_1, \lambda(dx))$, with  $\lambda(dx) = \lambda(d\vartheta, d\varphi) = d\varphi\, d\vartheta\, \sin \vartheta$ being $x \in \mathbf{S}^2_1$ represented as
\begin{equation*}
x = (\sin \varphi \cos \vartheta, \sin \vartheta \cos \varphi, \cos \vartheta), \quad \vartheta \in  [0, \pi], \; \varphi \in [0, 2\pi].
\end{equation*}
Sometimes, we will also write $f(x)=f(\vartheta, \varphi)$. The random coefficients  \eqref{alm-coeff-intro} are zero-mean Gaussian complex random variables such that (\cite{Baldi-Marinucci-2007})
\begin{equation}
\mathbb{E}[a_{lm}a^*_{l^\prime m^\prime}] = \delta_l^{l^\prime}\delta_m^{m^\prime} C_l = \mathbb{E} | a_{lm} |^2 \label{angular-power-C}
\end{equation}
where $C_l$, $l \geq 0$ is the {\bf angular power spectrum} of the random field $T$ which fully characterizes the dependence structure of $T$ (under Gaussianity) and
\begin{equation}
\delta_{a}^{b} = \left\lbrace \begin{array}{l}
1, \quad a=b\\
0, \quad a \neq b
\end{array} \right.
\end{equation}
is the Kronecker's delta symbol. The symbol ''*'' stands for complex conjugation. The interested reader can find a deep discussion and presentation of results concerning this field in the book by Marinucci and Peccati  \cite{DomPec-book}.

For a fixed integer $l$ and $\mu_l= l(l+1)$, the spherical harmonics
\begin{equation*}
Y_{lm}(\vartheta, \varphi) = \sqrt{\frac{2l+1}{4 \pi} \frac{(l-m)!}{(l+m)!}} Q_{lm}(\cos \vartheta) e^{im\varphi}
\end{equation*}
(or linear combination of them) solve the eigenvalue problem
\begin{equation} 
\triangle_{\mathbf{S}_{1}^2} Y_{lm}= - \mu_l \, Y_{lm} \label{eigenY}, \quad l \geq 0, \; |m| \leq l
\end{equation}
where
\begin{align}
\triangle_{\mathbf{S}_{1}^2} = &  \frac{1}{\sin^2 \vartheta} \frac{\partial^2}{\partial \varphi^2} + \frac{1}{\sin \vartheta} \frac{\partial}{\partial \vartheta} \left( \sin \vartheta \frac{\partial}{\partial \vartheta} \right) , \quad \vartheta \in [0,\pi],\; \varphi \in [0, 2\pi], \label{spherical-laplace}
\end{align}
is the spherical Laplace operator. We also recall that 
\begin{equation*}
Q_{lm}(z)=(-1)^m (1-z^2)^{m/2}\frac{d^m}{d z^m}Q_{l}(z)
\end{equation*}
are the associated Legendre functions and the Rodrigues' formula
$$ Q_l(z) = \frac{1}{2^l l!}\frac{d^l}{dz^l} (z^2 - 1)^l $$
defines the Legendre polynomials.

We introduce the transition function 
\begin{equation}
\mathbb{P}_t f(x - x_0) =\mathbb{E}^x f(B^{\Psi}_t - x_0) = \sum_{l=0}^{\infty} \sum_{m=-l}^{+l} f_{lm}\, Y_{lm}(x)Y^*_{lm}(x_0) \widetilde{P_t^\Psi}(\mu_l)  \label{transition-law-B-sub}
\end{equation}
where 
\begin{equation}
\widetilde{P_t^\Psi}(\mu) = \exp\left(-t \Psi(\mu) \right) = \mathbb{E}\exp\left(- \mu D_t \right) \label{lap-sub-intro}
\end{equation}
is the Laplace transform of the law of the subordinator $D_t$, $t>0$ with Laplace exponent $\Psi$. Formula \eqref{transition-law-B-sub} with $\Psi(\mu)=\mu$ (and therefore $D_t=t$, the elementary subordinator)  is the solution of the Cauchy problem
\begin{equation}
\label{Cauchy-rotational-BM}
\begin{array}{l}
(\partial_t - \triangle_{\mathbf{S}^2_1})u =0\\
u_0=f(x-x_0)
\end{array}
\end{equation}
and, for $f=\delta$ (the Dirac delta function, this means $f_{lm} \equiv 1$ for all $l,m$) is the transition law 
\begin{equation}
u(x, x_0, t) = \mathbb{E}u_0(x + B_t - x_0) = \sum_{l=0}^\infty \frac{2l+1}{4\pi} Q_l(\langle x,x_0\rangle) e^{-t \mu_l}\label{law-standard-rot-B}
\end{equation}
of the rotational Brownian motion $B_t$, $t>0$, with values in $\mathbf{S}^2_1$ and starting point $B_0=x_0 \in \mathbf{S}^2_1$. In formula \eqref{law-standard-rot-B} we have used the addition formula \eqref{addition-formula} below where
\begin{equation*}
\langle x,y \rangle = \cos d(x,y)
\end{equation*}
is the inner product in $\mathbf{R}^3$ and $d(\cdot, \cdot)$ is the usual spherical distance. Thus, formula \eqref{transition-law-B-sub} is the spherical convolution of $f$  with the transition law \eqref{law-standard-rot-B} of the subordinate rotational Brownian motion $B^\Psi_t=B_{D_t}$, $t>0$, with $B^\Psi_0=x_0$. 

From now on, we will consider the subordinate process $B^\Psi_t - x_0$ with law \eqref{transition-law-B-sub} and $f_{lm}=1$ for all $l,m$.
 
\begin{ass}
\label{ass1}
For all $t_1, t_2 > 0$ and $x,y \in \mathbf{S}^2_1$, the subordinate rotational Brownian motions $B^{\Psi}_{t_1}-x$ and $B^{\Psi}_{t_2}-y$ are independent if and only if $x \neq y$. 
\end{ass}
\begin{os}
We notice that this assumption is rational and justified by the fact that for different starting points, the processes $B^\Psi_{t_1}$ and $B^\Psi_{t_2}$ are two independent copies of $B^\Psi_{t}$ whereas, if $x=y$, then this means that $B^\Psi_{t_1}$ and $B^\Psi_{t_2}$ are two realizations of $B^\Psi_{t}$ at different times $t_1$ and $t_2$.
\end{os}

\subsubsection{Subordinate stochastic sphere}

The main object we deal with in this work is the time-dependent random field
\begin{equation}
\mathfrak{T}_t^{\Psi}(x) =  T(B^\Psi_t-x),\quad t>0,\; x \in \mathbf{S}^2_1 \label{T-Psi-field}
\end{equation}
which is the spherical random field $T$ indexed by the subordinate rotational Brownian motion $B^\Psi_t$, $t>0$, starting form $x \in \mathbf{S}^2_1$.  We say that the set $\{  B^\Psi_t, \; t>0 \} \subset \mathbf{S}^2_1$ is a subordinate stochastic sphere and represents the new set of indices by means of which the random field $T$ is observed.  

Let $\mathfrak{F}_{B^\Psi}$ be the $\sigma$-field generated by $B^\Psi_t$. We point out that (\cite{DovNan1})
\begin{align}
\mathbb{E}[\mathfrak{T}^\Psi_t(x) ]^n = & \mathbb{E}\big[ \mathbb{E}[T(B^\Psi_t - x) \big| \mathfrak{F}_{B^\Psi} \big] = \mathbb{E} \left[ \mathbb{E}[T(gx)]^n \Big| \mathfrak{F}_{B^\Psi} \right]  = \mathbb{E}[T(gx)]^n \label{moment-T-sub}
\end{align}
for all $g \in SO(3)$ due to the fact that (\cite{DomPec-book})
\begin{equation*}
\mathbb{E}[T(gx)]^n =  \sum_{l_1\ldots l_n} \sqrt{\frac{(2l_1+1)\cdots (2l_n+1)}{(4\pi)^n}} \mathbb{E}[a_{l_10}\cdots a_{l_n 0}] < \infty, \quad g \in SO(3)
\end{equation*}
(which does not depend on $x$) where $\mathbb{E}[a_{l_10}\cdots a_{l_n 0}]$ is called the angular polyspectrum of order $n-1$ associated with the field $T$. We also recall that 
\begin{equation}
\mathbb{E}[T(x)]^2 = \sum_{l=0}^\infty \frac{2l+1}{4\pi} C_l < \infty
\end{equation}
and $T$ is square integrable.

\begin{te}
\label{thm1}
{\bf (Space-Time covariance)} For $0 \leq t_1 \leq t_2 \leq \infty$ and $x,y \in \mathbf{S}^2_1$, we have that
\begin{align}
\mathbb{E}[\mathfrak{T}^{\Psi}_{t_1} (x)\, \mathfrak{T}^{\Psi}_{t_2}(y) ] = & \sum_{l=0}^\infty \frac{2l+1}{4\pi} C_l\, Q_l(\langle x, y \rangle) \widetilde{P_{t_1+t_2}^\Psi}(\mu_l), \quad x \neq y  \label{res-thm1}
\end{align}
where $C_l$ is the  angular power spectrum of the random field $T$. 
\end{te}

We recall that the angular power spectrum of $T$ is given by \eqref{angular-power-C} from which we recover the covariance
\begin{align}
\mathbb{E}[T(x)T(y)] = & \sum_{l=0}^\infty C_l \sum_{m=-l}^{+l} Y_{lm}(x)Y^*_{lm}(y) = \sum_{l=0}^{\infty} \frac{2l+1}{4\pi} C_l Q_l(\langle x, y \rangle). \label{cov-T}
\end{align}
In the last step we have considered once again the addition formula \eqref{addition-formula}.  Let us write the space covariance \eqref{res-thm1} as 
\begin{equation}
\gamma_t = \gamma_t(x, y) = \mathbb{E}[\mathfrak{T}^{\Psi}_{t_1} (x)\, \mathfrak{T}^{\Psi}_{t_2}(y) ], \quad \textrm{ for } \;  x \neq y, \; 0 < t = t_1+t_2 < \infty. \label{space-cov}
\end{equation} 
We can immediately verify that 
\begin{equation*}
\lim_{t \to 0^+} \gamma_t(x,y) = \gamma_0(x,y) = \mathbb{E}[T(x)T(y)]
\end{equation*}
and
\begin{equation*}
\lim_{t \to \infty} \gamma_t(x,y) = \gamma_{\infty}= \lim_{t \to \infty} \gamma_t(x,y) = \frac{C_0}{4\pi}.
\end{equation*}
Indeed, $\gamma_{\infty}$ immediately follows from \eqref{lap-sub-intro} whereas, from the representation \eqref{Trep} and formula \eqref{angular-power-C} we get \eqref{cov-T}. Therefore, the time-dependent random field \eqref{T-Psi-field} is the Gaussian random field $T$ at time zero and becomes an uniformly distributed random field as time passes and goes to infinity.

\begin{te}
\label{thm2}
{\bf (Time covariance)} For $0 \leq t_1 \leq t_2 \leq \infty$ and $x,y \in \mathbf{S}^2_1$, we have that 
\begin{align}
\mathbb{E}[\mathfrak{T}^{\Psi}_{t_1} (gx)\, \mathfrak{T}^{\Psi}_{t_2}(gx) ] = & \sum_{l=0}^{\infty} \frac{2l+1}{4\pi} C_l \, \widetilde{P_{t_2 - t_1}^\Psi}(\mu_{l}), \quad \forall\, g \in SO(3). \label{res-thm2}
\end{align}
\end{te}

\begin{os}
The variance of  $\mathfrak{T}^{\Psi}_{t} (x)$ can be obtained from \eqref{res-thm2} by observing that $\widetilde{P_{0}^\Psi}(\mu_{l})=1$ for all $l \geq 0$. Also we observe that for the covariance 
\begin{equation}
\gamma_\tau = \gamma_\tau(x,x) = \mathbb{E}[\mathfrak{T}^{\Psi}_{t_1} (x)\, \mathfrak{T}^{\Psi}_{t_2}(x) ], \quad 0 < \tau = t_2-t_1 < \infty
\end{equation}
we get that $\gamma_\infty = C_0/4\pi$. 
\end{os}

\begin{os}
We notice that $0 \leq \widetilde{P_{t}^\Psi}(\mu_l) \leq 1$.
\end{os}

We say that the zero-mean process $X_t$, $t\geq 0$, exhibits long range dependence if 
\begin{equation*}
\sum_{\tau=1}^\infty \mathbb{E}[X_{t+\tau} X_t] = \infty.
\end{equation*}
Otherwise, we have short range dependence. In the following remarks we focus on the  high frequency behaviour and the dependence structure of the random field considered so far. Since 
\begin{equation*}
L^2(\mathbf{S}^2_1) = \bigoplus_{l=0}^{\infty} \mathcal{H}_l
\end{equation*} 
the random field 
\begin{equation}
\mathfrak{T}^\Psi_{l,t}(x) = \sum_{m=-l}^{+l} a_{lm} Y_{lm}(B^\Psi_t - x), \quad x \in \mathbf{S}^2_1, \; t \geq 0, \; l \geq 0 \label{high-freq-comp}
\end{equation}
represents the projection of $\mathfrak{T}^\Psi_t(x)$ on $\mathcal{H}_l$ for each $l\geq 0$ and,  the high resolution analysis (for large $l$) of $\mathfrak{T}^\Psi_t(x)$ reveals interesting properties of $\mathfrak{T}^\Psi_t(x)$. As $l$ increases we get more and more information about observations. We say that  \eqref{high-freq-comp} is the  high frequency component of $\mathfrak{T}^\Psi_t(x)$. Furthermore, (see formula \eqref{cov-fun-proof} below)
\begin{equation}
\mathbb{E}[\mathfrak{T}^\Psi_{t_1}(x) \mathfrak{T}^\Psi_{t_2}(x)] = \sum_{l,l^\prime=0}^\infty  \mathbb{E}[\mathfrak{T}^\Psi_{l,t_1}(x) \mathfrak{T}^\Psi_{l^\prime, t_2}(x)] = \sum_{l=0}^\infty  \mathbb{E}[\mathfrak{T}^\Psi_{l,t_1}(x) \mathfrak{T}^\Psi_{l, t_2}(x)]
\end{equation}
which means that \eqref{high-freq-comp} are uncorrelated over $l$.

\begin{os}
\label{rem-stable}
Let $\Psi$ be the symbol of a stable subordinator, $\Psi(\mu)=\mu^\alpha$. We have that
\begin{align*}
\sum_{\tau =1}^{\infty} \mathbb{E}[\mathfrak{T}^\Psi_{l,t+\tau}(x) \mathfrak{T}^\Psi_{l,t}(x)] = & \frac{2l+1}{4\pi} C_l \sum_{\tau =1}^{\infty} e^{-\tau \mu_l^\alpha} \\ 
= & \frac{2l+1}{4\pi} C_l \frac{1}{e^{\mu_l^\alpha} -1}\\
 \approx & \frac{2l+1}{4\pi}C_l e^{- l^{2\alpha}}
\end{align*}
for large $l$. We have the same behaviour if we consider a stable subordinator with drift, that is $\Psi(\mu)=b\mu + \mu^\alpha$ with $b>0$. 
\end{os}

\begin{os}
\label{rem-geom-stable}
Let $\Psi$ be the symbol of a geometric stable subordinator, $\Psi(\mu)=\ln(1 + \mu^\alpha)$. We have that
\begin{align*}
 \sum_{\tau =1}^{\infty} \mathbb{E}[\mathfrak{T}^\Psi_{l,t+\tau}(x) \mathfrak{T}^\Psi_{l,t}(x)]  = & \frac{2l+1}{4\pi}C_l \sum_{\tau =1}^{\infty} \frac{1}{(1+\mu_l^\alpha)^\tau}\\
 = & \frac{2l+1}{4\pi} C_l \frac{1}{\mu_l^\alpha} \\
 \approx & \frac{2l+1}{4\pi}C_l l^{-2\alpha}
\end{align*}
for large $l$. For $\alpha=1$, the geometric stable subordinator becomes the gamma subordinator.
\end{os}

\begin{os}
\label{remark-class-D}
The power spectrum is commonly assumed to be $C_l \approx l^{-\gamma}$ for $\gamma>2$ to ensure summability (in \eqref{cov-T} for instance). Our power spectrum, say $\tilde{C}_l$, takes the form $C_l e^{-l^{2\alpha}}$ or $C_l l^{-2\alpha}$ as shown in the previous remarks and therefore, introduces a class of models in terms of power spectrum of Gaussian field (a slightly modification of the class $\mathfrak{D}$ introduced in \cite{DomPec-matphys-2010}) in which, for $\theta \in \mathbf{R}$, $\nu \in [0, \infty)$,
\begin{equation}
\sum_{l = L}^\infty l^{-\theta} e^{-l^\nu} < \infty, \quad L \textrm{ large}.
\end{equation} 
We notice that $\nu$ and $\theta$ must be such that if $\nu=0$, then $\theta > 2$. In \cite{DomPec-matphys-2010}, the authors have considered the class $\mathfrak{D}$ in which they  found connections between high frequency ergodicity and high frequency Gaussianity (for large $l$). 
\end{os}

\begin{os}
Under the assumption that $C_l \approx l^{-\gamma}$, $\gamma>2$, the random fields considered in Remark \ref{rem-stable} and Remark \ref{rem-geom-stable} exhibit   short range dependence but the rate of convergence of the sum over $l$ is different. This means that the $l$th frequency component of $\mathfrak{T}^\Psi_t(x)$ has different covariance structure as $l \to \infty$ for different symbol $\Psi$.
\end{os}

\begin{os}
Let us consider the sum of an $\alpha$-stable subordinator, say $X_t$ and, a geometric $\beta$-stable subordinator, say $Y_t$. We assume that $X_t$ and $Y_t$ are independent. For $c,d\geq 0$, the subordinator $X_{ct}+Y_{dt}$, $t>0$ has the symbol 
$$\Psi(\mu) = c\,\mu^\alpha +d\, \ln (1+ \mu^\beta), \quad \alpha, \beta \in (0,1).$$
The corresponding covariance function of $\mathfrak{T}^\Psi_t(x)$ leads to
\begin{align*}
 \sum_{\tau =1}^{\infty} \mathbb{E}[\mathfrak{T}^\Psi_{l,t+\tau}(x) \mathfrak{T}^\Psi_{l,t}(x)]  = & \frac{2l+1}{4\pi}C_l \sum_{\tau =1}^{\infty} \left(\frac{e^{-c\mu_l^\alpha}}{(1+\mu_l^\beta )^d} \right)^\tau\\
 = & \frac{2l+1}{4\pi} C_l \frac{1}{e^{c\mu_l^\alpha} + \mu_l^{d \beta} e^{c \mu_l^\alpha}-1} \\
 \approx & \frac{2l+1}{4\pi}C_l l^{-2d \beta} e^{- c l^{2\alpha}}.
\end{align*}
We immediately recognize the class we dealt with in Remark \ref{remark-class-D}.
\end{os}

In \cite{DovNan1} we have considered the time-changed Brownian motion on the sphere $B_{\mathcal{L}^\nu_t}$, $t>0$, $\nu \in (0,1)$ as coordinates change for the random field $T$. The inverse process $\mathfrak{L^\nu_t}$ can be regarded as the hitting time of a stable subordinator, say $\mathfrak{H}^\nu_t$, and therefore written as (\cite{DovECP, Nane-Surv})
\begin{equation*}
\mathfrak{L}^\nu_t = \inf \{ s \geq 0 \, : \, \mathfrak{H}^\nu_s \notin (0,t) \}.
\end{equation*}
The resulting field is related to equations involving fractional derivatives in time of order $\nu$ and, it has been shown that, its covariance structure reveals long range dependence for the high frequency components. In particular, the covariance function of the $l$th frequency component is written as
\begin{equation*}
\frac{2l+1}{4\pi} C_l E_{\nu}(-\mu_l \tau^\nu) \approx \frac{2l+1}{4\pi} C_l \frac{1}{\mu_l \tau^{\nu}} 
\end{equation*}
for large $l$, where $E_\nu$ is the Mittag-Leffler function. As we can immediately see, the sum over $\tau$ of the covariance above is infinite because of $\nu \in (0,1)$.

\subsubsection{Mean subordinate stochastic sphere}
We now introduce the zero-mean Gaussian random field
\begin{equation}
\eta^\Psi_t(x) = \mathbb{E}\left[\mathfrak{T}^\Psi_t(x)\big| \mathfrak{F}_T \right] \label{eta-psi-field}
\end{equation}
where $\mathfrak{F}_T$ is the $\sigma$-field generated by the random field $T$ on $\mathbf{S}^2_1$. We say that the random field $T$ is indexed by a mean subordinate stochastic sphere meaning that \eqref{eta-psi-field} can be written as (see formula \eqref{mean-Ylm-B-psi} below)
\begin{align*}
\eta^\Psi_t(x) = & \sum_{l=0}^{\infty} \sum_{|m| \leq l} a_{lm} \mathbb{E}[Y_{lm}(B^\Psi_t - x)] \\
= & \sum_{l=0}^{\infty} \sum_{|m| \leq l} a_{lm} e^{-t \Psi(\mu_l)} Y_{lm}(x) \\
= & \sum_{l=0}^\infty  e^{-t \Psi(\mu_l)} T_l(x), \quad x \in \mathbf{S}^2_1,\; t \geq 0.
\end{align*}
Furthermore,  the covariance function is written as (see formula \eqref{cov-eta-calculation} below)
\begin{align*}
\mathbb{E}[\eta^\Psi_{t_2}(x) \eta^\Psi_{t_1}(y)] = & \sum_{l=0}^{\infty} \frac{2l+1}{4\pi} C_l Q_l(\langle x, y \rangle) \, e^{-(t_2+t_1) \Psi(\mu_l)}, \quad x,y \in \mathbf{S}^2_1, \; t_1,t_2\geq 0 
\end{align*}
which coincides with \eqref{res-thm1}. Since the covariance above is defined for all $x,y$, we can write, for $x=y$,  
\begin{equation}
\mathbb{E}[\eta^\Psi_t(gx)]^2 =  \sum_{l=0}^{\infty} \frac{2l+1}{4\pi} C_l e^{-2t \Psi(\mu_l)}, \quad \forall\, g \in SO(3) \label{var-eta-psi}
\end{equation}
which does not coincide with $\mathbb{E}[\mathfrak{T}^\Psi_t(x)]^2$ in \eqref{cov-T} as  it is in \eqref{moment-T-sub} for $\mathfrak{T}^\Psi_t(x)$. The variance \eqref{var-eta-psi} can be also obtained from the higher-order formula
\begin{equation}
\mathbb{E}[\eta_t^\Psi(gx)]^n = \sum_{l_1 \cdots l_n} e^{-t \Psi(\mu_{l_j})} \sqrt{\frac{\prod_{j=1}^n(2l_j+1)}{(4\pi)^n}} \mathbb{E}[a_{l_1 0} \cdots a_{l_n 0}]
\end{equation} 
for all $g \in SO(3)$ and $n \in \mathbf{N}$, which also differs from the higher-order moments of $T$. As we can see, in this case Assumption \ref{ass1} fails and the random field \eqref{eta-psi-field} is qualitatively different from \eqref{T-Psi-field}. Since $\Psi \geq 0$ ($\Psi$ is a Berstein function), we notice that $\mathbb{E}[\eta^\Psi_t(x)]^n \to 0$ as $t \to \infty$ for all $n$. The Fourier random coefficients are obviously time-dependent and such that
\begin{align*}
\mathbb{E}[a_{lm}(t_2) a^*_{l^\prime m^\prime}(t_1) ] = \delta_l^{l^\prime} \delta_m^{m^\prime} e^{-(t_1+t_2) \Psi(\mu_l)} \mathbb{E}|a_{lm}|^2 = \delta_l^{l^\prime} \delta_m^{m^\prime} e^{-(t_1+t_2) \Psi(\mu_l)} C_l
\end{align*}
which comes directly from \eqref{alm-coeff-intro}. We can also write
\begin{equation}
\mathbb{E}|a_{lm}(t)|^2 = e^{- 2t \Psi(\mu_l)} C_l, \quad l\geq 0, \; m=-l, \ldots , +l. \label{alm-t-coef}
\end{equation}
From the Parseval's identity we get that
\begin{equation*}
\int_{\mathbf{S}^2_1} |\eta^\Psi_t(x)|^2 \lambda(dx) = \sum_{l=0}^\infty \sum_{m=-l}^{+l} |a_{lm}|^2 e^{-2t \Psi(\mu_l)} = \sum_{l=0}^\infty\sum_{m=-l}^{+l} |a_{lm}(t)|^2 < \infty
\end{equation*}
or equivalently we say that $\eta^\Psi_t(x) \in L^2(d\lambda \otimes dP)$. 

Let $f$ be a square integrable function on the unit sphere, $f \in L^2(\mathbf{S}^2_1)$. For $\alpha \in (0,1)$, we introduce the fractional operator
\begin{equation}
\mathbb{D}_{M}^\alpha f(x) = \int_0^\infty \left( P_s\, f(x) - f(x) \right) M(ds) \label{frac-oper-sphere}
\end{equation}
where $M(\cdot)$ is the L\'{e}vy measure introduced in Section \ref{SectionAuxProof} below and $P_t = e^{t\triangle_{\mathbf{S}^2_1}}$ is the transition semigroup of the rotational Brownian motion $B_t$. As we will show below, the operator \eqref{frac-oper-sphere} can be rewritten as
\begin{equation}
\mathbb{D}^\alpha_M\, f(x) = \int_{\mathbf{S}^2_1}\left( f(y) - f(x) \right) J(x,y)\lambda(dy) 
\end{equation}
where $\lambda$ is the Lebesgue measure on $\mathbf{S}^2_1$ and 
\begin{equation*}
J(x,y) = \sum_{l=0}^\infty \frac{2l+1}{4\pi} Q_l(\langle y, x \rangle) \Psi^\prime(\mu_l)
\end{equation*}
with $\Psi^\prime(\mu) = d/d\mu \Psi(\mu)$. 

The next results are devoted to the study of the solutions of certain partial differential equations with random initial condition. Seminal works in this field are those by  Kravvaritis \cite{Rand-cond-eq-1} and Kamp\'{e} de F\'{e}riet \cite{Rand-cond-eq-2}. The literature concerning random solutions of partial differential equations is vaste and therefore we skip a further introduction and jump to the following results.  

\begin{te}
\label{thm-seq-Psi}
Let $\Psi$ be the symbol of a subordinator with no drift. For $\alpha \in (0,1)$, the solution to
\begin{equation}
\label{eq-eta-psi-field}
\left\lbrace
\begin{array}{l}
(\partial_t - \mathbb{D}_{M}^\alpha )  \eta_t^\Psi(x) =0\\
\eta^\Psi_0(x)=T(x)
\end{array} \right .
\end{equation}
is a Gaussian random field on $\mathbf{S}^2_1$ which can be written as
\begin{equation}
\eta_t^\Psi(x) = \sum_{l=0}^\infty\sum_{m=-l}^{+l} a_{lm}(t) Y_{lm}(x) = \sum_{l=0}^\infty e^{-t \Psi(\mu_l) }  T_l(x) \label{u-field-rep-Psi}
\end{equation}
where
\begin{equation}
a_{lm}(t) = a_{lm} \exp\left(-t \Psi(\mu_l) \right), \quad t>0
\end{equation}
and the series converges in the sense that for all $t \geq 0$
\begin{equation}
\lim_{L \to \infty} \mathbb{E}\left[ \int_{\mathbf{S}^2_1} \left( \eta_t^\Psi(x) -  \sum_{l=0}^L e^{-t \Psi(\mu_l) }  T_l(x) \right)^2 \lambda(dx) \right] = 0. \label{covergence-series}
\end{equation}
Furthermore, 
\begin{equation}
\eta^\Psi_t(x-x_0) = \mathbb{P}_t T(x-x_0) = \mathbb{E}[ \mathfrak{T}^\Psi_t(x-x_0) \big| \mathfrak{F}_T ] 
\end{equation}
where $\mathfrak{F}_T$ is the $\sigma$-field generated by $T$.
\end{te}

\begin{os}
Let $\Psi$ be the symbol of a subordinator with drift $b \geq 0$. The problem \eqref{eq-eta-psi-field} becomes
\begin{equation}
\left\lbrace
\begin{array}{l}
(\partial_t - b \triangle_{\mathbf{S}^2_1} - \mathbb{D}_{M}^\alpha )  \eta_t^\Psi(x) =0\\
\eta^\Psi_0(x)=T(x)
\end{array} \right .
\end{equation}
and the random solution is
\begin{equation}
\eta_t^\Psi(x) = \sum_{l=0}^\infty e^{-t \Psi(\mu_l) }  T_l(x) .
\end{equation}
where $\Psi$ is written as in formula \eqref{lap-exp-sub} below.
\end{os}

We now study the special case $\Psi(\mu)=\mu^\alpha$, that is to consider a subordinate rotational Brownian motion with a random time which is a stable subordinator. The corresponding random field $\mathfrak{T}^\Psi_t(x)$ solves a stochastic fractional equation involving the fractional spherical Laplacian with random initial condition given by the random field $T(x)$, $x \in \mathbf{S}^2_1$.
\begin{coro}
\label{thm-seq}
For $\alpha \in (0,1)$, the solution to
\begin{equation}
\left\lbrace
\begin{array}{l}
(\partial_t + (-\triangle_{\mathbf{S}^2_1})^\alpha )  \eta_t(x) =0\\
\eta_0(x)=T(x)
\end{array} \right .
\end{equation}
is a Gaussian random field on $\mathbf{S}^2_1$ which can be written as
\begin{equation}
\eta_t(x) = \sum_{l=0}^\infty  e^{-t \mu^\alpha_l }  \sum_{|m| \leq l} a_{lm} Y_{lm}(x) = \sum_{l=0}^\infty  e^{-t \mu^\alpha_l }  T_l(x) \label{u-field-rep}
\end{equation}
in the sense of \eqref{covergence-series}. Furthermore, 
\begin{equation}
\eta_t(x-x_0) = \mathbb{P}_t T(x-x_0) = \mathbb{E}[ \mathfrak{T}^\Psi_t(x-x_0) \big| \mathfrak{F}_T ] 
\end{equation}
where $\mathfrak{F}_T$ is the $\sigma$-field generated by $T$ and $\Psi(\mu)=\mu^\alpha$.
\end{coro}

The definition of fractional powers of operators is given below and intended in the sense of Bochner's subordination rule (see for example \cite{Komatzu}). This is an interesting topic with a long history, we refer to \cite{DovSPA} and especially to the references therein for alternative definitions.

\begin{os}
Let the symbol $\Psi$ be as in Remark \ref{remark-class-D}. For a fixed time $t\geq 0$, say $t=1$ for the sake of simplicity, the coefficients \eqref{alm-t-coef} define an angular power spectrum, say $\tilde{C}_l$, which takes the form
\begin{equation}
\tilde{C}_l = C_l \, e^{- (c\mu^\alpha + d\ln (1+\mu^\beta))} \approx C_l \, l^{-2 d \beta } e^{-c l^{2\alpha}}\label{pow-spec-doppio-sub} 
\end{equation}
for large $l$. In particular, we obtain a large class of random fields with angular power spectrum  \eqref{pow-spec-doppio-sub} and parameters $\beta, \alpha \in [0,1]$ such that
\begin{equation*}
\sum_{l=L}^\infty \tilde{C}_t < \infty, \quad L \, \textrm{ large}.
\end{equation*}
In the analysis of the CMB radiation for instance, in modelling the Sachs-Wolfe effect we have a polynomial decay of the angular power spectrum whereas, the Silk damping effect entails an exponential decay as $l$ increases.
\end{os}

\begin{coro}
\label{coro-ultimo}
Let us consider $\Psi(\mu)=\mu^\alpha$, $\alpha \in (0,1)$ in \eqref{space-cov}. We have that
\begin{equation}
(\partial_t + (-\triangle_{\mathbf{S}^2_1})^\alpha)\gamma_t(x, \cdot) = 0 = (\partial_t + (-\triangle_{\mathbf{S}^2_1})^\alpha)\gamma_t(\cdot, y) 
\end{equation}
in both space variables $x$ and $y$,  subject to the initial condition 
$$\gamma_0(x, y) = \mathbb{E}[T(x)T(y)]$$
where $x,y \in \mathbf{S}^2_1$ with no restriction. 
\end{coro}

\section{Auxiliary results and proofs}
\label{SectionAuxProof}
Let us consider the L\'{e}vy process $\mathbf{F}_t$, $t>0$, with associated Feller semigroup $P_t \, f(\mathbf{x}) = \mathbb{E}f(\mathbf{x} + \mathbf{F}_t)$ solving $\partial_t u = \mathcal{A}u$ with initial datum $u_0=f$. For the infinitesimal generator $\mathcal{A}$ of $\mathbf{F}_t$, $t>0$ the following representation holds
\begin{equation}
(\mathcal{A}f)(\mathbf{x}) = -\frac{1}{(2\pi)^d} \int_{\mathbf{R}^d} e^{-i \boldsymbol{\xi}\cdot \mathbf{x}} \Phi(\boldsymbol{\xi}) \widehat{f}(\boldsymbol{\xi})\boldsymbol{d \xi} \label{inv-fourier-symbol}
\end{equation}
for all functions in the domain 
\begin{equation}
D(\mathcal{A}) = \left\lbrace f \in L^2(\mathbf{R}^d,\mathbf{dx})\,:\, \int_{\mathbf{R}^d} \Phi(\boldsymbol{\xi}) |\widehat{f}(\boldsymbol{\xi})|^2 \boldsymbol{d \xi}< \infty \right\rbrace
\end{equation}
where $\widehat{f}(\boldsymbol{\xi}) = \int_{\mathbf{R}^d} e^{i \boldsymbol{\xi}\cdot \mathbf{x}} f(\mathbf{x}) \mathbf{dx}$ is the Fourier transform of $f$, $\Phi(\cdot)$ is continuous and negative definite. We say that $P_t$ is a pseudo-differential operator with Fourier symbol $\widehat{P_t} = \exp(-t\Phi)$ and, $\Phi$ is the Fourier symbol of $\mathcal{A}$, that is $\widehat{(\mathcal{A}f)}(\boldsymbol{\xi}) = -\Phi(\boldsymbol{\xi}) \widehat{f}(\boldsymbol{\xi})$. L\'{e}vy processes with
\begin{equation}
\Phi(\xi) = ib\xi + \int_0^\infty \left( e^{i\xi y } - 1  \right) M(dy)
\end{equation} 
where $b\geq 0$ and the L\'{e}vy measure $M(\cdot)$ satisfies $\int_0^\infty (y \wedge 1)M(dy) < \infty$ and $M(-\infty, 0)=0$, are the so-called subordinators and possesses non-negative increments, that is non-decreasing paths. 
The Laplace transform of the law of  a subordinator leads to the Laplace exponent
\begin{equation}
\Psi(\mu) = - \Phi(i\mu) = b\mu + \int_0^\infty \left( 1 - e^{-\mu y} \right) M(dy) \label{lap-exp-sub}
\end{equation}
for each $\mu >0$ and therefore, we can write $\widetilde{P^\Psi_t}(\mu) = \exp\left( -t \Psi(\mu) \right)$. We write below some explicit forms of the Laplace exponent $\Psi$ and the measure $M$ of the corresponding subordinator ($\alpha \in (0,1)$):
\begin{itemize}
\item $\Psi(\mu) = \mu^\alpha$: stable, $M(y) = \alpha y^{-\alpha -1}/ \Gamma(1-\alpha)$;
\item $\Psi(\mu) = b\mu + \mu^\alpha$: stable with drift, $M(\cdot)$ as above and $b>0$ in \eqref{lap-exp-sub};
\item $\Psi(\mu) = \ln (1 + \mu )$: gamma, $M(y) = y^{-1}e^{-y}$;
\item $\Psi(\mu) = \ln (1+ \mu^\alpha)$: geometric stable, $M(y)=\alpha y^{-1}E_\alpha(-y)$ where 
\begin{equation*}
E_\alpha(z) = \sum_{j=0}^\infty \frac{z^j}{\Gamma(\alpha j + 1)}
\end{equation*}
is the Mittag-Leffler function.
\end{itemize}

Let $\mathfrak{H}^\alpha_t$, $t>0$ be a stable subordinator with $\alpha \in (0,1)$. The generator of $\mathbf{F}_{\mathfrak{H}^\alpha_t}$, $t>0$ is given by the beautiful formula
\begin{equation}
-(-\mathcal{A})^\alpha f(x) = \frac{\alpha}{\Gamma(1-\alpha)} \int_0^\infty \Big( P_s \, f(x) - f(x) \Big) \frac{ds}{s^{\alpha +1}} \label{beautiful}
\end{equation}
for all $f \in \mathscr{S}$ (the space of rapidly decaying $C^\infty$ functions) and we formally write the Feller semigroup of $\mathbf{F}_t$, $t>0$ as $P_t=e^{t\mathcal{A}}$ which must be understood in the sense of functional calculus. For $\mathcal{A}=-\frac{\partial}{\partial x}$, $P_t$ is the translation semigroup (or shift operator in the context of operational calculus) and formula \eqref{beautiful} becomes the Riemann-Liouville fractional derivative governing $\alpha$-stable subordinators. It is well known that, for $\Phi(\boldsymbol{\xi})=|\boldsymbol{\xi} |^\alpha$, formula \eqref{inv-fourier-symbol} gives us the fractional power of the Laplace operator which can be also expressed as 
\begin{align}
-(-\triangle_{\mathbf{R}^d})^{\alpha} f(\mathbf{x}) =  & C_d(\alpha) \, \textrm{p.v.} \int_{\mathbf{R}} \frac{f(\mathbf{x}+\mathbf{y}) -f(\mathbf{x})}{|\mathbf{y}|^{2\alpha +d}}\mathbf{dy}, \quad \mathbf{x} \in \mathbf{R}^d \label{lap-princ-val}
\end{align}
where ''p.v.'' stands for the ''principal value'' of the singular integrals above near the origin and $C_d(\alpha)$ is a normalizing constant depending on $d$ and $\alpha$. The subordinate Brownian motion on the sphere is obtained by considering the rotational Brownian motion on $\mathbf{S}^2_1$ with a randomly varying time given by $\mathfrak{H}^\alpha_t$. The governing operator is the fractional power of the spherical Laplace operator \eqref{spherical-laplace}, that is $-(-\triangle_{\mathbf{S}^2_1} )^{\alpha}$, $\alpha \in (0,1)$ or $P_s =\exp( s \triangle_{\mathbf{S}^2_1})$ in formula \eqref{beautiful}. 

Hereafter, we use also the symbols
\begin{equation*}
\sum_{lm} = \sum_l \sum_m = \sum_l \sum_{|m| \leq l} = \sum_{l=0}^\infty \sum_{m=-l}^{+l}
\end{equation*}
in order to streamline the notation as much as possible.

First we discuss on the solution to \eqref{Cauchy-rotational-BM}. Let  $f \in L^2(\mathbf{S}^2_1)$.  Since $\{Y_{lm}:l\in \mathbb{N},  m\in \mathbb{Z} \ \mathrm{and}\ |m|\leq l\}$ is dense in $L^2(\mathbf{S}^2_1)$, the function $f$ (see for example the Peter-Weyl representation theorem on the sphere in \cite{DomPec-book} or,  the references therein) can be written as
\begin{equation*}
f(x)=f(\vartheta, \varphi) = \sum_{lm} f_{lm} Y_{lm}(\vartheta, \varphi)
\end{equation*}
which holds in the $L^2$ sense, where
\begin{equation}
f_{lm}=\int_{\mathbf{S}^2_1} f(\vartheta, \varphi) Y^*_{lm}(\vartheta, \varphi)\, \sin \vartheta\, d\vartheta\, d\varphi . \label{f-coeff-exp}
\end{equation}
The {\bf angular power spectrum} of $f$ is defined as
\begin{equation}\label{def-angular-spectrum}
A_l = A_l (f) = \sum_{|m|\leq l } | f_{lm}|^2.
\end{equation}
By following the same arguments we get
\begin{equation}
f(x-x_0) = \sum_{lm} f_{lm}Y_{lm}(x) Y^*_{lm}(x_0). \label{f-x-x0}
\end{equation}
The strongly continuous semigroup $P_t$ associated to $\triangle_{\mathbf{S}^2_1}$ is written as
\begin{align*}
P_t\, f(x - x_0) = & \mathbb{E}^x f(B_t - x_0) = \sum_{lm}  f_{lm}Y_{lm}(x)Y^*_{lm}(x_0) \exp\left( -t \mu_l \right)
\end{align*}
where $x_0 \in \mathbf{S}^2_1$ is the starting point at $t_0=0$ of a rotational Brownian motion $B_t$, $t>0$ (or Brownian motion on the unit sphere $\mathbf{S}^2_1$) and comes out from 
\begin{align*}
\mathbb{E}^x f(B_t - x_0) = & \sum_{l^\prime m^\prime} \sum_{lm} f_{lm} Y_{l^\prime m^\prime}(x) Y^*_{lm}(x_0) \int_{\mathbf{S}^2_1} Y^*_{l^\prime m^\prime}(y) Y_{lm}(y) \lambda(dy)\, \exp\left( -t \mu_l \right)
\end{align*}
by considering \eqref{f-x-x0} and the orthogonality property of spherical harmonics
\begin{equation}
\int_{\mathbf{S}^2_1} Y_{lm}(x)\, Y_{l^\prime m^\prime}^{*}(x)\, \lambda(dx)  = \delta_{l}^{l^\prime}\, \delta_{m}^{m^\prime} \label{orthoY}
\end{equation}
where $Y^*_{lm}= (-1)^mY_{l-m}$ is the complex conjugation of $Y_{lm}$. For the sake of clarity we notice that, according to \eqref{transition-law-B-sub}, $ P_t = \mathbb{P}_t$ for $\Psi(\mu)=\mu$.

\begin{os}
We observe that:
\begin{enumerate}
\item $\mathbb{P}_0=Id$,
\item $\mathbb{P}_t 1= \mathbb{P}_t 1_{\mathbf{S}^2_1} = 1$,
\item $\mathbb{P}_t\, \mathbb{P}_s = \mathbb{P}_{t+s}$.
\end{enumerate}
Indeed, $(1)$ is easy to check, see \eqref{f-x-x0}; $(2)$ follows from \eqref{orthoY}; $(3)$ is a consequence of the Bochner subordination rule and the fact that the law of $B_t$ satisfies the Chapman-Kolmogorov equation. Set $f_{lm}=1$ for all $l,m$, that is $f=\delta$ and $\mathbb{P}_t \delta$ is the transition function of $B^\Psi_t$, $t>0$. For all $x, y \in \mathbf{S}^2_1$ (reproducing kernel) 
\begin{equation}
 \int_{\mathbf{S}^2} Q_{l}(\langle x,z \rangle)\, Q_{l^\prime}(\langle z, y \rangle)\, \lambda(dz) = \delta_{l}^{l^\prime} \frac{4\pi}{2l+1} Q_{l}(\langle x, y \rangle) \label{reproducK}
\end{equation}
and $\mathbb{P}_t \delta(x-x_0)$ satisfies the Chapman-Kolmogorov equation.
\end{os}

The convergence of the series expansion \eqref{Trep} is understood in the following sense
\begin{equation*}
\lim_{L\rightarrow \infty }\mathbb{E}\left[ \int_{\mathbf{S}^{2}}\left( T(x) - \sum_{l=0}^{L} \sum_{m=-l}^{+l} a_{lm} Y_{lm}(x) \right)^{2} \lambda (dx) \right] =0
\end{equation*}
and also
\begin{equation*}
\lim_{L\rightarrow \infty }\mathbb{E}\left(T(x)-\sum_{l=0}^{L}\sum_{m=-l}^{+l}a_{lm}Y_{lm}(x)\right)^{2}= 0
\end{equation*}
as pointed out in the introduction. The coordinates change random field is defined as follows
\begin{equation}
\mathfrak{T}_t^{\Psi}(x) = \sum_{lm} a_{lm}Y_{lm}(B^\Psi_t - x), \quad x \in \mathbf{S}^2_1, \; t >0 \label{T-Psi-rep}
\end{equation}
where the subordinate rotational Brownian motion is a measurable map from $(\Omega, \mathfrak{F}_{B^\Psi}, P)$ to $(\mathbf{S}^2_1, \mathscr{B}(\mathbf{S}^2_1), \lambda)$. Therefore, the convergence of \eqref{T-Psi-rep} must be understood in the sense that for all $t\geq 0$
\begin{equation*}
\lim_{L\rightarrow \infty } \mathbb{E}\left[ \mathbb{E} \Bigg[ \int_{\mathbf{S}^{2}}\left( T(B^\Psi_t - x) - \sum_{l=0}^{L} \sum_{m=-l}^{+l} a_{lm} Y_{lm}(B^\Psi_t - x) \right)^{2} \lambda (dx)  \Bigg| \mathfrak{F}_{B^\Psi} \Bigg] \right] =0
\end{equation*}
and also, $\forall\, t$
\begin{equation*}
\lim_{L\rightarrow \infty } \mathbb{E}\left[ \mathbb{E}\left(T(B^\Psi_t - x)-\sum_{l=0}^{L}\sum_{m=-l}^{+l}a_{lm}Y_{lm}(B^\Psi_t - x)\right)^{2} \Bigg| \mathfrak{F}_{B^\Psi} \right]= 0.
\end{equation*}
From  \eqref{Trep} we write
$$\mathfrak{T}^{\Psi}_t(x) = \sum_{l \geq 0} \mathfrak{T}^{\Psi}_{l,t}(x)$$
where
$$\mathfrak{T}^{\Psi}_{l,t}(x) = \sum_{|m| \leq l} a_{lm}Y_{lm}(B^{\Psi}_t -x), \quad t>0,\; l \geq 0$$
is a random field representing the $l$th frequency component of $\mathfrak{T}^{\Psi}_t$ or the projection of $\mathfrak{T}^{\Psi}_t$ on $\mathcal{H}_l$ which is the space generated by $\{Y_{lm}: |m| \leq l \}$. We have that
\begin{align*}
\mathbb{E}[\mathfrak{T}^{\Psi}_{l_1,t_1}(x) \mathfrak{T}^{\Psi}_{l_2,t_2}(y)] = &  \sum_{|m_1|\leq l_1} \sum_{|m_2| \leq l_2} \mathbb{E}[a_{l_1m_1}a^*_{l_2m_2}] \, \mathbb{E}[Y_{l_1m_1}(B^{\Psi}_{t_1}-x)\, Y^*_{l_2m_2}(B^\Psi_{t_2}-y)]. 
\end{align*}
and the Fourier random coefficients $a_{lm}$ are independent Gaussian complex random variables for which \eqref{angular-power-C} holds. Thus,  we get that
\begin{align*}
\mathbb{E}[\mathfrak{T}^{\Psi}_{l_1,t_1}(x) \mathfrak{T}^{\Psi}_{l_2,t_2}(y)] = & \delta_{l_1}^l\delta_{l_2}^l \sum_{|m|\leq l}  C_{l} \, \mathbb{E}[Y_{lm}(B^{\Psi}_{t_1}-x)\, Y^*_{lm}(B^{\Psi}_{t_2}-y)]
\end{align*}
and therefore, the covariance function of $\mathfrak{T}^{\Psi}_{t} (x)$ can be written as follows
\begin{align}
\mathbb{E}[\mathfrak{T}^{\Psi}_{t_1} (x)\, \mathfrak{T}^{\Psi}_{t_2}(y) ] = & \sum_{l \geq 0} \mathbb{E}[\mathfrak{T}^{\Psi}_{l,t_1}(x) \mathfrak{T}^{\Psi}_{l,t_2}(y)], \quad x \neq y  \label{cov-fun-proof}
\end{align}
and we recall that $C_l$, $l \geq 0$ is the angular power spectrum of the random field $T$.

\begin{proof}[Proof of Theorem \ref{thm1}]
We have that $x \neq y$, then for the Assumption \ref{ass1}, we write \eqref{cov-fun-proof} as follows 
\begin{align*}
\mathbb{E}[\mathfrak{T}^{\Psi}_{l_1,t_1}(x) \mathfrak{T}^{\Psi}_{l_2,t_2}(y)] = &  \delta_{l_1}^l\delta_{l_2}^l \sum_{|m|\leq l}  C_l \, \mathbb{E}[Y_{lm}(B^{\Psi}_{t_1}- x)]\, \mathbb{E}[Y^*_{lm}(B^{\Psi}_{t_2}-y)]
\end{align*}
where, from the fact that
\begin{equation}
\frac{2l^\prime +1}{4\pi} \int_{\mathbf{S}^2_1} Q_{l^\prime}(\langle x,y \rangle) Y_{lm}(y) \lambda(dy) = \delta_l^{l^\prime} Y^*_{lm}(x),
\end{equation} 
we obtain that
\begin{align}
\mathbb{E}[Y_{lm}(B^{\Psi}_{t_1}- x)] = & \sum_{l^\prime} \frac{2l^\prime +1}{4\pi} \int_{\mathbf{S}^2_1} Y_{lm}(y)  Q_{l^\prime}(\langle y, x \rangle) \lambda(dy) \widetilde{P_{t_1}^\Psi}(\mu_{l^\prime}) = Y^*_{lm}(x) \widetilde{P_{t_1}^\Psi}(\mu_{l}) \label{mean-Ylm-B-psi}
\end{align}
which is as to write $f_{l m} = \delta_l^{l^\prime} \delta_m^{m^\prime}$ in \eqref{f-coeff-exp} with $f$ as in formula \eqref{transition-law-B-sub}. From the fact that $Y^*_{lm} = (-1)^m Y_{l-m}$, we also obtain 
\begin{align*}
\mathbb{E}[Y^*_{lm}(B^{\Psi}_{t_2} - y)] = & (-1)^m \mathbb{E}[Y_{l -m}(B^{\Psi}_{t_2} - y)]\\
= & (-1)^m Y^*_{l-m}(x) \widetilde{P_{t_2}^\Psi}(\mu_{l})\\
= & Y_{lm}(y) \widetilde{P_{t_2}^\Psi}(\mu_{l}).
\end{align*}
Thus, we get that
\begin{align*}
\mathbb{E}[\mathfrak{T}^{\Psi}_{l_1,t_1}(x) \mathfrak{T}^{\Psi}_{l_2,t_2}(y)] = &  \delta_{l_1}^l\delta_{l_2}^l \, C_l \, \widetilde{P_{t_1+t_2}^\Psi}(\mu_l) \sum_{|m|\leq l}   Y_{lm}(x)Y^*_{lm}(y)
\end{align*}
and, from the addition formula
\begin{equation}
\sum_{|m| \leq l} Y_{lm}(x)Y^*_{lm}(y) =  \frac{2l+1}{4\pi} Q_l(\langle x, y \rangle) \label{addition-formula}
\end{equation}
we arrive at
\begin{align}
\mathbb{E}[\mathfrak{T}^{\Psi}_{l_1,t_1}(x) \mathfrak{T}^{\Psi}_{l_2,t_2}(y)] = & \delta_{l_1}^l\delta_{l_2}^l \frac{2l+1}{4\pi} C_l \,\widetilde{P_{t_1+t_2}^\Psi}(\mu_l)\, Q_l(\langle x, y \rangle) \label{cov-proof1}
\end{align}
which is the (space-time) covariance of the $l$th frequency component of $\mathfrak{T}^{\Psi}_t$. By combining \eqref{cov-proof1} with \eqref{cov-fun-proof} we obtain the claim.
\end{proof}

\begin{proof}[Proof of Theorem \ref{thm2}]
We start once again from the fact that
\begin{align*}
\mathbb{E}[\mathfrak{T}^{\Psi}_{l_1,t_1}(x) \mathfrak{T}^{\Psi}_{l_2,t_2}(y)] = & \delta_{l_1}^l\delta_{l_2}^l \sum_{|m|\leq l}  C_{l} \, \mathbb{E}[Y_{lm}(B^{\Psi}_{t_1}-x)\, Y^*_{lm}(B^{\Psi}_{t_2}-y)]
\end{align*}
as we have shown in the previous proof. Here we have that $x=y$ and therefore, from Assumption \ref{ass1}, we get that
\begin{align*}
& \mathbb{E}[Y_{lm}(B^{\Psi}_{t_1}- x)\,Y^*_{lm}(B^{\Psi}_{t_2}-x)] \\
= & \sum_{l^\prime} \frac{2l^\prime +1}{4\pi} \int_{\mathbf{S}^2_1} \int_{\mathbf{S}^2_1}  Q_{l^\prime}(\langle y_1, y_2 \rangle ) Y_{lm}(y_1) Y_{lm}(y_2) \lambda(dy_1)\lambda(dy_2) \widetilde{P_{t_2 - t_1}^\Psi}(\mu_{l^\prime})\\
= & \int_{\mathbf{S}^2_1} Y_{lm}(y_2) Y_{lm}(y_2) \lambda(dy_2) \widetilde{P_{t_2 - t_1}^\Psi}(\mu_{l})\\
= & \widetilde{P_{t_2 - t_1}^\Psi}(\mu_{l}).
\end{align*}
Thus, formula \eqref{cov-fun-proof} takes the form
\begin{align*}
\mathbb{E}[\mathfrak{T}^{\Psi}_{t_1} (x)\, \mathfrak{T}^{\Psi}_{t_2}(x) ] = & \sum_{l} \frac{2l+1}{4\pi} C_l \, \widetilde{P_{t_2 - t_1}^\Psi}(\mu_{l})
\end{align*}
and this concludes the proof.
\end{proof}

We now consider the  random field on the mean subordinate stochastic sphere
\begin{equation}
\eta^\Psi_t(x) = \mathbb{E}[\mathfrak{T}^\Psi_t(x) | \mathfrak{F}_T] =  \mathbb{E}[T(B^\Psi_t - x) | \mathfrak{F}_T], \quad x \in \mathbf{S}^2_1,\; t\geq 0 \label{eta-proof}
\end{equation}
where $T$ is the Gaussian random field \eqref{Trep} and $\mathfrak{F}_T$  is the $\sigma$-field generated by $T$. From \eqref{Trep} and \eqref{angular-power-C} we can write the covariance function of \eqref{eta-proof} as
\begin{align}
\mathbb{E}[\eta^\Psi_{t_1}(x)\eta^\Psi_{t_2}(y)] =  & \mathbb{E} \left[\sum_{lm} \sum_{l^\prime m^\prime} a_{lm} a^*_{l^\prime m^\prime} e^{-t_1 \Psi(\mu_l) - t_2 \Psi(\mu_{l^\prime})} Y_{lm}(x)Y^*_{l^\prime m^\prime}(y) \right]\notag \\
= & \sum_{lm} C_l e^{-(t_1+t_2) \Psi(\mu_l)} Y_{lm}(x)Y^*_{l m}(y)\notag \\
= & \sum_{l} C_l e^{-(t_1+t_2) \Psi(\mu_l)} \sum_{|m| \leq l} Y_{lm}(x)Y^*_{lm}(y)\notag \\
= & \sum_l \frac{2l+1}{4\pi} C_l e^{-(t_1+t_2) \Psi(\mu_l)}Q_l(\langle x, y \rangle) \label{cov-eta-calculation}
\end{align}
for all $x,y \in \mathbf{S}^2_1$ and $t_1,t_2\geq 0$.

\begin{prop}
For $n \in \mathbf{N}$, the higher-order moments of \eqref{eta-proof} are given by
\begin{equation*}
\mathbb{E}[\eta_t^\Psi(gx)]^n = \sum_{l_1 \cdots l_n} e^{-t \Psi(\mu_{l_j})} \sqrt{\frac{\prod_{j=1}^n(2l_j+1)}{(4\pi)^n}} \mathbb{E}[a_{l_1 0} \cdots a_{l_n 0}], \quad \forall\, g \in SO(3).
\end{equation*} 
\end{prop}

\begin{proof}
The higher-order moments of \eqref{eta-proof} can be obtained as follows
\begin{align*}
\mathbb{E}[\eta_t^\Psi(x)]^n = & \mathbb{E} \left[ \sum_{l_1 \cdots l_n} \prod_{j=1}^n T_{l_j}(x) e^{-t \Psi(\mu_{l_j})} \right]\\
= & \sum_{l_1 \cdots l_n} \mathbb{E}\left[ \prod_{j=1}^n T_{l_j}(x) \right] e^{-t \Psi(\mu_{l_j})}
\end{align*}
where
\begin{align*}
\mathbb{E}\left[ \prod_{j=1}^n T_{l_j}(x) \right] = & \sum_{m_1 \cdots m_n} \mathbb{E}[a_{l_1m_1} \cdots a_{l_n m_n}] \prod_{j=1}^n Y_{l_j m_j}(x).
\end{align*}
The random field $T$ is isotropic. Due to the fact that $T_l(x) \stackrel{law}{=}T(x_N)$ where $x_N=(0,0)$ is the North Pole and that $Y_{lm}(x_N) =0$ for $m \neq 0$ and $Y_{l0}(x_N) = \sqrt{(2l+1)/4\pi}$ (\cite{Quantum}), we can write
\begin{align*}
\mathbb{E}\left[ \prod_{j=1}^n T_{l_j}(x) \right] = & \sqrt{\frac{\prod_{j=1}^n(2l_j+1)}{(4\pi)^n}} \mathbb{E}[a_{l_1 0} \cdots a_{l_n 0}] .
\end{align*}
By collecting all pieces together we get
\begin{equation}
\mathbb{E}[\eta_t^\Psi(x)]^n = \sum_{l_1 \cdots l_n} e^{-t \Psi(\mu_{l_j})} \sqrt{\frac{\prod_{j=1}^n(2l_j+1)}{(4\pi)^n}} \mathbb{E}[a_{l_1 0} \cdots a_{l_n 0}].
\end{equation} 
By observing that 
\begin{equation*}
\mathbb{E}[\eta_t^\Psi(gx)]^n = \mathbb{E}[\eta_t^\Psi(x)]^n, \quad \forall\, g \in SO(3)
\end{equation*}
we complete the proof.
\end{proof}

The  fractional operator \eqref{frac-oper-sphere} can be rewritten as
\begin{align*}
\mathbb{D}^\alpha_M\, f(x) = & \int_0^\infty \left( P_sf(x) - f(x) \right) M(ds)\\
= & \int_0^\infty \left( \mathbb{E}f(x+B_s) - f(x) \right) M(ds)\\
= & \int_0^\infty \mathbb{E}\left[\left( f(x+B_s) - f(x) \right) \right] M(ds)\\
= & \int_0^\infty \int_{\mathbf{S}^2_1}\left( f(y) - f(x) \right) \lambda(y) Pr\{ x+ B_s \in dy \}  M(ds)\\
= &  \int_{\mathbf{S}^2_1}\left( f(y) - f(x) \right) J(x,y)\lambda(dy) 
\end{align*}
where
\begin{align*}
J(x,y) = & \int_0^\infty Pr\{ x+B_s \in dy \}/dy  M(ds)\\
= & \sum_{l} \frac{2l+1}{4\pi} Q_l(\langle y, x \rangle) \int_0^\infty e^{-s \mu_l} M(ds)\\
= & \sum_{l} \frac{2l+1}{4\pi} Q_l(\langle y, x \rangle) \Psi^\prime(\mu_l)
\end{align*}
and $\Psi^\prime(\mu) = d/d\mu \Psi(\mu)$ as we can see from \eqref{lap-exp-sub} with $b=0$. We notice that
\begin{equation*}
0 \leq J(x,y) < \infty\quad \textrm{ and } \quad J(x,y)=J(y,x).
\end{equation*}
We also recall that $\Psi$ is a Bernstein function and therefore $\Psi \geq 0$ (in particular, $\Psi(0)=0$) and $(-1)^k d^k/d \mu^k\Psi(\mu) \leq 0$ for all $\mu \geq 0$ and $k=1,2, \ldots$. 

\begin{proof}[Proof of Theorem \ref{thm-seq-Psi}]
First we show that
\begin{equation}
\mathbb{P}_s Y_{lm}(x) = Y_{lm}(x) e^{-s \Psi(\mu_l)}. \label{semigroup-B-proof-psi}
\end{equation}
Indeed, we have that
\begin{align*}
 \mathbb{P}_t Y_{lm}(x) = &  \mathbb{E} Y_{lm}(x+B^\Psi_t)\\
= &  \sum_{l^\prime m^\prime} \widetilde{P_t}(\mu_{l^\prime}) Y_{l^\prime m^\prime}^*(x) \int_{\mathbf{S}^2_1} Y_{lm}(y) Y_{l^\prime m^\prime }(y)\lambda(dy)\\
= & \widetilde{P_t^\Psi}(\mu_{l}) (-1)^m Y_{l -m}^*(x)\\
= &\widetilde{P_t^\Psi}(\mu_{l})  Y_{l m}(x)
\end{align*}
were we have used the orthogonality property \eqref{orthoY} and the fact that $Y_{lm} = (-1)^m Y^*_{l-m}$. We recall that
\begin{equation*}
\widetilde{P_t^\Psi}(\mu) = \mathbb{E}\exp\left( -\mu B_{D_t} \right) = \exp\left( -t \Psi(\mu) \right)
\end{equation*}
where $D_t$ is a subordinator with Laplace exponent $\Psi$. For $\Psi(\mu)=\mu$, we get that
\begin{equation}
P_s Y_{lm}(x) = Y_{lm}(x) e^{-s \mu_l} \label{semigroup-B-proof}
\end{equation}
where $P_t = e^{t \triangle_{\mathbf{S}^2_1}}$ is the transition semigroup of the rotational Brownian motion $B_t$, $t>0$, with transition law \eqref{law-standard-rot-B}. We assume that \eqref{u-field-rep-Psi} holds true. We immediately see that $\eta^\Psi_0(x) = T(x)$ which is in accord with the initial condition. With \eqref{semigroup-B-proof} in mind, we obtain 
\begin{align*}
\mathbb{D}_{M}^\alpha \eta^\Psi_t(x) = & \int_0^\infty \left( P_s\, \eta^\Psi_t(x) - \eta^\Psi_t(x) \right) M(ds)\\
= & \sum_{lm} e^{-t \Psi(\mu_l)} a_{lm} \int_0^\infty \left( P_s\, Y_{lm}(x) - Y_{lm}(x) \right) M(ds)\\
= & \sum_{lm} e^{-t \Psi(\mu_l)} a_{lm} \int_0^\infty \left( e^{-s \mu_l} Y_{lm}(x) - Y_{lm}(x) \right) M(ds)\\
= & \sum_{lm} e^{-t \Psi(\mu_l)} a_{lm} Y_{lm}(x) \int_0^\infty \left( e^{-s \mu_l} - 1 \right) M(dy)\\
= & \sum_l e^{-t \Psi(\mu_l)} \left( - \int_0^\infty \left( 1 - e^{-s \mu_l} \right) M(ds) \right) T_l(x)\\
= & \sum_l e^{-t \Psi(\mu_l)} \left( -\Psi(\mu_l) \right) T_l(x)
\end{align*}
where in the last step we have considered \eqref{lap-exp-sub} with drift term $b=0$. Indeed, we are considering subordinators with no drift. Also, we have that
\begin{equation}
\sum_l \partial_t e^{-t \Psi(\mu_l)} T_l(x) = \sum_l e^{-t \Psi(\mu_l)} \left( -\Psi(\mu_l) \right) T_l(x) < \infty.
\end{equation}
From this, we arrive at
\begin{equation}
(\partial_t - \mathbb{D}_{M}^\alpha )  \eta_t^\Psi(x) = \sum_{lm}a_{lm} (\partial_t - \mathbb{D}_{M}^\alpha ) e^{-t\Psi(\mu_l)} Y_{lm}(x) = 0
\end{equation}
end equation \eqref{eq-eta-psi-field} is satisfied.

For the completeness of the set $\{Y_{lm}\,:\, l\geq 0,\; m=-l, \ldots , +l\}$, the series \eqref{u-field-rep-Psi} converges in the sense that, $\forall\, t$,
\begin{align*}
\lim_{L \to \infty}\mathbb{E}\left[ \int_{\mathbf{S}^2_1} \left( \eta^\Psi_t(x) - \sum_{l=0}^L e^{-t \Psi(\mu_l)} T_l(x)  \right)^2 \lambda(dx) \right] = &  0.
\end{align*}
Indeed, if we write 
\begin{equation*}
S_L(x) = \sum_{l=0}^N e^{-t \Psi(\mu_l)} T_l(x) = \sum_{l=0}^N \sum_{|m|\leq l}e^{-t \Psi(\mu_l)} Y_{lm}(x)
\end{equation*}
for a fixed $t>0$, then we obtain that
\begin{align*}
\mathbb{E}\| \eta^\Psi_t - S_L\|^2 = & \mathbb{E}\| \eta_t^\Psi \|^2 - \sum_{l=0}^{L} \sum_{|m| \leq l} \mathbb{E}|a_{lm}|^2 e^{-2t\Psi(\mu_l)}\\
 = & \mathbb{E}\| \eta_t^\Psi \|^2 - \sum_{l=0}^{L} \sum_{|m| \leq l} (2l+1)\,C_l e^{-2t\Psi(\mu_l)}
\end{align*} 
where
\begin{equation*}
\| \eta_t^\Psi \|^2 = \int_{\mathbf{S}^2_1} | \eta_t^\Psi(x) |^2 \lambda(dx) \quad \textrm{and} \quad \mathbb{E}\| \eta_t^\Psi \|^2 = \lambda(\mathbf{S}^2_1) \mathbb{E}[\eta^\Psi_t(x)]^2.
\end{equation*}
From \eqref{var-eta-psi}, by recalling that $\lambda(\mathbf{S}^2_1)=4\pi$ and taking the limit for $L \to \infty$, we get that
\begin{align*}
\lim_{L \to \infty} \mathbb{E}\| \eta^\Psi_t - S_L\|^2 = 0.
\end{align*}
Furthermore, from \eqref{semigroup-B-proof-psi}, 
\begin{align*}
\mathbb{P}_t T(x-x_0) = & \sum_l \mathbb{P}^\Psi_t T_l(x-x_0) = \sum_{lm}a_{lm} \mathbb{P}_t Y_{lm}(x-x_0)\\
= &\sum_{lm} a_{lm}\widetilde{P_t^\Psi}(\mu_{l})  Y_{l m}(x-x_0)\\
= & \sum_l \widetilde{P_t^\Psi}(\mu_{l}) T_l(x-x_0) = \sum_{l} e^{-t\Psi(\mu_l)}T_l(x-x_0).
\end{align*}
The proof is completed.
\end{proof}

\begin{proof}[Proof of Corollary \ref{thm-seq}]
We use formula \eqref{beautiful}.  We assume that \eqref{u-field-rep} holds true and, from \eqref{semigroup-B-proof}, we get that 
\begin{align*}
-(-\triangle_{\mathbf{S}^2_1})^\alpha  \eta_t(x) = & \frac{\alpha}{\Gamma(1-\alpha)} \int_0^\infty \Big( P_s \, \eta_t(x) - \eta_t(x) \Big) \frac{ds}{s^{\alpha +1}}\\
= & \frac{\alpha}{\Gamma(1-\alpha)} \sum_{lm} \widetilde{P_t^\Psi}(\mu_l) a_{lm} \int_0^\infty \Big( P_s \, Y_{lm} (x) - Y_{lm}(x) \Big) \frac{ds}{s^{\alpha +1}}\\
= &  \sum_{l} \widetilde{P_t^\Psi}(\mu_l) \left( \frac{\alpha}{\Gamma(1-\alpha)} \int_0^\infty \Big( e^{-s\mu_l} - 1 \Big) \frac{ds}{s^{\alpha +1}}\right) T_l(x)\\
= & \sum_{l} \widetilde{P_t^\Psi}(\mu_l) (- \mu_l^\alpha ) T_l(x)\\
= &  \sum_{l} \partial_t\, \widetilde{P_t^\Psi}(\mu_l) T_l(x).
\end{align*}
We also observe that $\widetilde{P_0^\Psi}(\mu)=1$ and therefore we get $\eta_0(x)$ as in the claim. 
\end{proof}

Now we prove our last result.

\begin{proof}{Proof of Corollary \ref{coro-ultimo}}
It follows from the proof of Corollary \ref{thm-seq}. From \eqref{u-field-rep} we can write

\begin{equation}
\mathbb{E}[\eta_{t_1}(x)\eta_{t_2}(y)] = \sum_{l=0}^\infty \sum_{l^\prime=0}^\infty  e^{-t_1 \mu^\alpha_l } e^{-t_2 \mu^\alpha_{l^\prime} }  \, \mathbb{E}[ T_l(x) T_{l^\prime}(y)]
\end{equation}
From \eqref{cov-T}, by setting $t_1=t_2=t/2$, we get that
\begin{equation}
\gamma_t(x,y) = \sum_{l=0}^\infty  e^{-t \mu^\alpha_l } \, \mathbb{E}[ T_l(x) T_{l}(y)]
\end{equation}
and therefore, 
\begin{align*}
0 = & (\partial_t + (-\triangle_{\mathbf{S}^2_1})^\alpha)\gamma_t(x, \cdot)\\
= & \mathbb{E}\left[ \sum_{l=0}^\infty  (\partial_t + (-\triangle_{\mathbf{S}^2_1})^\alpha) e^{-t \mu^\alpha_l } T_l(x) T_{l}(\cdot) \right]\\
= & \mathbb{E}\left[ \sum_{l=0}^\infty \sum_{m, m^\prime =-l}^{+l} a_{lm}a_{lm^\prime}  (\partial_t + (-\triangle_{\mathbf{S}^2_1})^\alpha) e^{-t \mu^\alpha_l } Y_{lm}(x) Y_{lm^\prime}(\cdot) \right].
\end{align*}
The same reasoning applies for
\begin{align*}
(\partial_t + (-\triangle_{\mathbf{S}^2_1})^\alpha)\gamma_t(\cdot, y) =& 0.
\end{align*}
\end{proof}

\end{document}